\documentclass[11pt]{article}
\usepackage[utf8]{inputenc}
\usepackage[T1]{fontenc}
\usepackage[left = 3cm, right = 3cm, top=3cm, bottom=3cm]{geometry}
\usepackage{setspace}
\usepackage{cellspace}
\usepackage[x11names]{xcolor}
\usepackage{mathtools,amssymb,mathrsfs}
\usepackage{amsmath}
\usepackage{stmaryrd}
\usepackage{calc, ifthen, xspace}
\usepackage{tikz,tkz-tab}
\usepackage{fancybox}
\usepackage{amsthm}
\usepackage{enumitem}
\usepackage{empheq}
\usepackage{graphicx}
\usepackage{dsfont}
\usepackage{mdwlist}
\usepackage[numbers]{natbib}
\usepackage{hyperref}

\author{Corentin Faipeur\footnote{Institut Fourier, Université Grenoble Alpes, France - corentin.faipeur@univ-grenoble-alpes.fr}}

\title{Subcritical Boolean percolation on graphs of bounded degree}
\newtheorem{thm}{Theorem}

\newtheorem{lemme}{Lemma}

\newtheorem{rem}{Remark}

\newcommand{\E}{\mathbf{E}}
\renewcommand{\P}{\mathbf{P}}

\newcommand{\R}{\mathbf{R}}

\newcommand{\Z}{\mathbf{Z}}
\newcommand{\N}{\mathbf{N}}

\newcommand{\ds}{\displaystyle}
\newcommand{\ind}[1]{\mathbf{1}_{\{#1\}}}
\newcommand{\ssi}{\Longleftrightarrow}

\newcommand{\et}{\text{ and }}
\newcommand{\si}{\text{ if }}

\newcommand{\fonction}[4]{\begin{array}{|rcl}
#1 & \longrightarrow & #2 \\
#3 & \longmapsto & #4 \end{array}}

\newcommand{\cas}[4]{\begin{cases}
#1 & #2\\
#3 & #4
\end{cases}}

\newcommand{\B}{\mathcal{B}}
\newcommand{\C}{\mathcal{C}}

\newcommand{\rad}{\operatorname{rad}}

\begin{document}
\maketitle
\begin{abstract}
    In this paper, we study a model of long-range site percolation on graphs of bounded degree, namely the Boolean percolation model. In this model, each vertex of an infinite connected graph is the center of a ball of random radius, and vertices are said to be active independently with probability $p\in [0,1]$. 
    We consider $W$ to be the reunion of random balls with an active center. In certain circumstances, the model does not exhibit a phase transition, in the sense that $W$ almost surely contains an infinite component for all $p>0$, or even $W$ covers the entire graph.
    In this paper, we give a sufficient condition on the radius distribution for the existence of a subcritical phase, namely a regime such that almost surely all the connected components of $W$ are finite.
    Additionally, we provide a sufficient condition for the exponential decay of the size of a typical component.
\end{abstract}

\section{Introduction}
We consider a discrete version of Boolean percolation on general graphs of bounded degree.
It can be described as follows. Each vertex of an infinite connected graph is considered as the center of a ball with random radius.
We assume that the radii are independent and identically distributed (i.i.d.) positive random variables.
Each vertex is also set to be \emph{active} with probability $p \in [0,1]$, independently of anything else.
Consider $W$ the reunion of random balls with an active center.
In the original paper by Lamperti \cite{lamperti}, the vertices are considered as fountains that wet their neighbours up to some distance (which is $0$ with probability $p$), this is why we call $W$ the \emph{wet set}.
Note that in our setting, an active fountain wets its own site (which is not the case in \cite{lamperti}).
Therefore, it is clear that $W$ is empty if $p=0$ and that $W$ covers the entire graph if $p=1$.
Under weak assumption on the graph, we can construct situations where $W$ almost surely contains an infinite connected component for all $p>0$. In fact, it is the case as soon as the graph contains an infinite path of vertices and the expected radius of a wet ball diverges.
On an undirected graph, we even have that $W$ almost surely covers the entire graph for any $p>0$ if the expected sizes of the wet balls are infinite.
In this paper, we aim to find sufficient conditions on the radius distribution for the existence of a subcritical phase, namely non-trivial values of $p$ such that all the connected components of $W$ are finite almost surely. We also consider the question of exponential decay of the size of a typical connected component.

This model can be seen as a discrete version of the continuum Poisson Boolean percolation, which has been widely studied (e.g. see the book \cite{book_continuum_perco}). In the continuum, for any $d\geq 1$, we consider random balls centered at the points of a homogeneous Poisson point process on $\R^d$ of intensity $\lambda \in (0,\infty)$.
Here $W$ is the subset of $\R^d$ of points covered by the balls, and $W_0$ denotes the connected component of the origin (possibly empty).
Let $R$ be a copy of one of the random radii. In \cite{Hall}, Hall proved that for any $d\geq 1$, if $\E[R^{2d-1}]$ is finite, then $W_0$ is almost surely bounded for small enough $\lambda$. This never happens when $\E[R^d]$ is infinite, because in this case the whole space is almost surely covered for any $\lambda>0$ (Proposition 3.1 in \cite{book_continuum_perco}).
Gou\'er\'e proved in \cite{gouere} that the finiteness of $\E[R^d]$ is actually a necessary and sufficient condition for the existence of a subcritical phase.
Sharpness of the phase transition has been investigated in \cite{ATT}, \cite{DCRT} and \cite{dembin2022sharpsharpnesspoissonboolean}.

Returning to the discrete setting, Coletti and Grynberg established in \cite{coletti2014absencepercolationbernoulliboolean} a result similar to that of Gou\'er\'e, proving that Boolean percolation on $\Z^d$ admits a subcritical phase if and only if $\E[R^d]$ is finite, for all $d \geq 1$. Together with Miranda \cite{ColettiMirandaGrynberg}, they furthermore prove an analogous statement for doubling graphs, which includes all the graphs with polynomial growth (thus $\Z^d$).
In \cite{Braun_2024}, Braun handled the case of the directed $d$-ary tree, for which the existence of an infinite component in $W$ is equivalent to the survival of a branching process.
A related model is the Rumor Percolation model (see the review \cite{review} and the references therein), which can be described as follows.
At time 0, only the root (a fixed vertex of the graph) knows about the rumor, and starts to spread it.
Then, at each step, a new vertex hears the rumor if it is within a random radius of a previously informed vertex.
We point out the paper by Lebensztayn and Rodriguez \cite{LEBENSZTAYN20082130} which deals with general graphs of bounded degree (but with geometric distribution for the radii) and shows a non-trivial phase transition.
However, contrary to the Boolean setting above, this rumor model (or disk-percolation in \cite{LEBENSZTAYN20082130}) is crucially directed.
In fact, the rumor necessarily starts from the root and then spreads across the graph, whereas in Boolean percolation the rumor may originate from a source at infinity and subsequently contaminate the root.
In this paper, we prove that for any graph with degree bounded by $\Delta>1$, the Boolean model admits a subcritical phase if $\E[\Delta^{2R}]$ is finite.
In fact, a weaker condition will be given in Theorem \ref{thm_pc>0}.
The graph can be directed or not, and the above condition can be improved for directed graphs. The proof of this result can be seen as a Peierls argument.
We then show that the size of the connected component of the root has exponential tails for $p$ small enough if $\E[e^{t \Delta^R}]$ is finite for some $t>0$.
This is obtained by comparison with a subcritical Galton-Watson process, so we do not assert that this result would be sharp.

\paragraph{Acknowledgments}
This work is part of a PhD thesis supervised by Vincent Beffara. I would like to thank him for his valuable advice and support, as well as for our stimulating discussions. I would also like to thank Jean-Baptiste Gou\'er\'e for his useful comments on the first version of this paper (in particular for pointing out a mistake in the proof of Lemma \ref{lemma_domination}), and the anonymous referee for providing additional bibliographical information.

\section{Notations and statement of the main results}

Throughout the paper, $\N = \{0,1,2 \dots\}$ denotes the set of non-negative integers and $\N^* = \{1,2 \dots\}$ the set of positive integers.
For an event $A$, $\mathbf{1}_A$ is the indicator function of the event.
For two probability measures $\mu$ and $\nu$, we write $\mu \preceq \nu$ when $\nu$ stochastically dominates $\mu$, i.e. when there exists a coupling $(X,Y)$ of the two measures such that $X\leq Y$ almost surely.

\subsection{Graph terminology}\label{Graphs terminology}

Let $G=(V,E)$ be an infinite locally finite connected graph. We will consider both directed and undirected graphs.
Thus, $E$ is seen as a subset of $V\times V$, and $G$ is undirected if and only if $\forall x,y \in V$, $(x,y)\in E \ssi (y,x) \in E$.
However, for directed graphs, we will only consider the weak connectivity, i.e. the connectivity of the underlying undirected graph.
For all $x,y \in V$, the distance $d(x,y)\in \N \cup \{+\infty\}$ from $x$ to $y$ is defined as the minimal number of (oriented) edges of a path from $x$ to $y$ (and $d(x,y)=+\infty$ if such a path does not exist).
Note that it is symmetric if and only if the graph is undirected.
For $x \in V$ and $r\in \N$, let $B_r(x) := \{y \in V \mid d(x,y) <r\}$ (so $B_0(x) = \emptyset$, $B_1(x) = \{x\}$ etc.). 

We now define three different notions of the boundary of a subset of the graph. For every $B \subset V$, let 
$$\partial^+B := \{y \in V-B \mid \exists x \in B, (x,y) \in E\},$$
$$\partial^-B := \{y \in V-B \mid \exists x \in B, (y,x) \in E\},$$
$$\partial B:= \partial^+B \cup \partial^-B.$$
Note that we only consider outer boundaries, meaning that a vertex in $\partial B$ do not belong to $B$ but it is adjacent (the orientation of the edges between it and the vertices of $B$ determine whether it is on $\partial^+B$, $\partial^-B$ or both). These three notions coincide for undirected graphs.
Associated with that, we define three distinct notions of degree. For every $x \in V$, let
$$\begin{tabular}{cccc}
    $d^+(x) := |\partial^+\{x\}|$,&$d^-(x) := |\partial^-\{x\}|$&\et & $d(x) := |\partial\{x\}|$;
\end{tabular}$$
$d^+(x)$ and $d^-(x)$ are usually called the out-degree and in-degree of vertex $x$, and they correspond respectively to the number of oriented edges emanating from $x$ or pointing at $x$.
However, $d(x)$ is not the total number of edges attached to $x$, but rather the number of neighbours of $x$ without regard to the orientation, i.e. $d(x)$ is the degree of $x$ in the underlying undirected graph. 
Finally, let $$\Delta_+=\Delta_+(G):= \sup \{d_+(x), x \in V\},$$ $$\Delta_-=\Delta_-(G):= \sup \{d_-(x), x \in V\},$$ $$\Delta = \Delta(G) :=\sup \{d(x), x \in V\};$$
we clearly have $\Delta_\pm \leq \Delta \leq \Delta_+ + \Delta_-$.
If $\Delta$ is finite, the graph is said to be of bounded degree.
Note again that if the graph is undirected, we have for all $x\in V$, $d^+(x)=d^-(x)=d(x)$ and $\Delta_+=\Delta_-=\Delta$.

For an oriented graph $G=(V,E)$, we define its \emph{reversed graph} or \emph{transpose graph} $G^\top$ on the same set of vertices, but with reversed orientation of the edges.
Therefore, if $(x,y) \in E$ is an edge of $G$, then $(y,x)$ is an edge of $G^\top$ and vice versa.
For $x \in V$ and $r\in \N$, we write $B^\top_r(x)$ for the corresponding ball in $G^\top$, i.e. $B^\top_r(x)$ is the set of vertices $y \in V$ such that $d(y,x) <r$.
In other words, for all $x,y \in V$, we have the equivalence
$$y \in B_r(x) \ssi x \in B_r^\top(y).$$

Suppose that $G$ is of bounded degree.
Define for all $r \in \N$
\begin{equation*}
    c_r=c_r(G):=\sup_{x \in V} |B_r(x)| \et c_r^\top = c_r^\top(G) := \sup_{x \in V} |B^\top_r(x)|.
\end{equation*}
Since the graph is of bounded degree, we know that these two quantities are finite.
In fact, $c_r \leq 1+\Delta_+ + \dots + \Delta_+^{r-1} \leq \max(r, 2 \Delta_+^{r-1})$ and similarly $c_r^\top \leq \max(r, 2 \Delta_-^{r-1})$.
These two quantities are also upper bounded by $2 \Delta^{r-1}$ (since $\Delta>1$ for any infinite connected graph).
We also define $s_r$ and $s^\top_r$ respectively as the supremum over $x \in V$ of the number of vertices at distance exactly $r$ from $x$ in $G$ or $G^\top$ respectively:
\begin{equation*}
    s_r=s_r(G):=\sup_{x \in V} |\partial^+ B_r(x)| \et s_r^\top = s_r^\top(G) := \sup_{x \in V} |\partial^- B^\top_r(x)|.
\end{equation*}
In the following, we will use the upper bounds $s_r \leq \Delta_+^r$ and $s_r^\top \leq \Delta_-^r$.

\subsection{Phase transition of the Boolean percolation model}
We now precisely define our Boolean model.
Let $p \in [0,1]$ and let $\nu$ be a probability distribution on $\N^*$. Let $(R_x)_{x\in V}$ be a family of i.i.d. random variables distributed according to $\nu$, and let $(\sigma_x)_{x\in V}$ be a family of independent Bernoulli random variables of parameter $p$, independent from $(R_x)_{x\in V}$.
A vertex $x$ of the graph should be seen as a fountain that attempts to wet vertices at distance less than $R_x$, including itself.
The attempt is successful if and only if $\sigma_x=1$.
Therefore, the set of wet vertices, or covered set, is defined as 
\begin{equation}\label{W(G)}
    W=W(G):= \bigcup_{x \in V, \sigma_x=1} B_{R_x}(x).
\end{equation}
The balls $B_{R_x}(x)$ with $\sigma_x=1$ that appear in the union \eqref{W(G)} are called \emph{wet balls}.
For all $x\in V$, write $W_x$ for the connected component of $x$ in $W$ (with the convention that $W_x=\emptyset$ when $x \notin W$).
As already observed, the probability of the event $\{|W_x| = \infty\}$ is 0 for $p=0$ and $1$ for $p=1$, and is clearly non-decreasing in $p$.
The \emph{percolation} event is defined as the reunion
$\bigcup_{x\in V} \{|W_x| = \infty\}$. It has probability 0 or 1 by the Kolmogorov's zero–one law.
Therefore, we define the critical parameter of the model as
\begin{align*}
    p_c(G, \nu) = \inf \{p \in [0,1] : \P\left(\cup_{x \in V} \{|W_x| = \infty\}\right) =1 \}.
\end{align*}

As already observed in \cite{coletti2014absencepercolationbernoulliboolean} and \cite{LEBENSZTAYN20082130}, we have $p_c(G, \nu) \leq p_c^{\text{site}}(G)$ where $p_c^{\text{site}}(G)$ is the critical parameter for Bernoulli site percolation on $G$.
In fact, $W(G)$ contains all the sites $x \in V$ such that $\sigma_x=1$, so the Boolean model stochastically dominates Bernoulli site percolation.

We prove the following result.

\begin{thm}\label{thm_pc>0}
    Let $G$ be a graph of bounded degree, and $R$ distributed according to $\nu$.
    For all $n \geq 1$, let $\varphi(n):= \Delta \sum_{r=1}^n c_r c_r^\top$.
    If $\E[\varphi(R)] < \infty$, then
    $$p_c(G,\nu) \geq \dfrac{1}{\E[\varphi(R)]} >0.$$
\end{thm}

%In the definition of $\varphi(n)$, we put $\max(\Delta s_{r-1}^\top, c_r^\top)$ to optimise the lower bound on $p_c(G, \nu)$, but note that $c_r^\top \leq \max(\Delta s_{r-1}^\top, c_r^\top) \leq \Delta c_r^\top$. Therefore, requiring $\E[\varphi(R)] < \infty$ is equivalent to $\E[\tilde \varphi(R)] < \infty$ with $\tilde \varphi(n)= \sum_{r=1}^n c_r c_r^\top$.
Using the inequalities given at the end of Section \ref{Graphs terminology}, observe that
$$\varphi(n) \leq \Delta \sum_{r=1}^n 2 \Delta^{r-1} 2 \Delta^{r-1} = \frac{4\Delta}{\Delta^2-1} (\Delta^{2n}-1).$$
Therefore, we immediately get from Theorem \ref{thm_pc>0} that $p_c(G,\nu) >0$ if $\E[\Delta^{2R}] < \infty$. This is the best condition we can obtain without further assumption on the graph.
On a directed graph, the condition can be replaced by $\E[\max(R, \Delta_+^R)\max(R, \Delta_-^R)] < \infty$.
\begin{rem}
    Consider $\mathbf{T}^d$ the infinite $(d+1)$--regular tree, and choose an orientation of its edges so that $\Delta_-=d$ and $\Delta_+=1$ (this can be done by choosing an end of the tree and orientating all the edges in the direction of this end).
    For this graph, the condition for having a phase transition given by Theorem \ref{thm_pc>0} is $\E[Rd^R]< \infty$. We do not claim that it is necessary, but one can easily observe that the finiteness of $\E[d^R]$ is necessary.
    It is a general fact in Boolean percolation models that for any site $\rho\in \mathbf{T}^d$, the expected number of wet balls containing $\rho$ is $p\sum_{x \in V} \sum_{r>d(x,\rho)} \nu(r)$.
    This is exactly $p (\E[d^R]-1)/(d-1)$ in the case of $\mathbf{T}^d$ with our choice of orientation. If this expectation diverges, by applying the Borel-Cantelli lemma we find that $\rho$ belongs to infinitely many wet balls, and finally $W=\mathbf{T}^d$ almost surely.
\end{rem}

Let $\rho \in V$ be a specified vertex of the graph that will be called the root.
We give a sufficient condition to ensure that the size of $W_\rho$ has an exponential tail for small enough $p$.

\begin{thm}\label{thm_expo_decay}
    For all $t>0$ and $n\geq 1$, let $\psi_t(n) = \sum_{r=1}^n c_r^\top e^{t c_r}$.
    Under the same assumption as in Theorem \ref{thm_pc>0}, if furthermore $\E[\psi_t(R)]$ is finite for some $t>0$, then for all $p<\E[\varphi(R)]^{-1}$, there exists $\lambda>0$ such that for every $n\in \N$,
    $$\P(|W_\rho| > n) \leq e^{-\lambda n}.$$
    In particular, $\E[\psi_t(R)] < \infty$ holds if $\E\left[e^{2t \Delta^R}\right] < \infty$.
\end{thm}

\subsection{Reformulation of the main results}
The wet set defined in \eqref{W(G)} can be rewritten as
$$W=\bigcup_{x \in V} B_{\sigma_x R_x} (x).$$
To be consistent with this reformulation, for all $n\geq 0$, let $q_n:=1-p \sum_{k>n} \nu(k)$ be the probability that $\sigma_x R_x$ is at most $n$.

We derive Theorems \ref{thm_pc>0} and \ref{thm_expo_decay} from proving the following result, stated in this new setting for directed graphs.

\begin{thm}\label{boolean_perco}
    Let $G$ be an infinite directed graph of bounded degree.
    \begin{enumerate}
        \item If $(q_n)$ satisfies
        \begin{equation}\label{condition_subcriticality}
            \Delta\sum_{n \geq 0} c_{n+1} c_{n+1}^\top (1-q_n) < 1,
        \end{equation}
    then almost surely all connected components of $W(G)$ are finite.
        \item Moreover, if \eqref{condition_subcriticality} holds and if for some $t>0$
        \begin{equation}\label{condition_expodecay}
            \sum_{n \geq 0} c^\top_{n+1} e^{tc_{n+1}} (1-q_n) < \infty,
        \end{equation}
    then there exists $\lambda>0$ such that for every $n \in \N$,
    $$\P(|W_\rho| > n) \leq e^{-\lambda n}.$$ 
    \end{enumerate}
\end{thm}
Conditions \eqref{condition_subcriticality} and \eqref{condition_expodecay} are those that arise from the proof in Section \ref{Proof of Theorem}. We verify that they are implied by $\E[\varphi(R)] < \infty$ and $\E[\psi_t(R)] < \infty$ respectively if $p$ is small enough.

\begin{proof}[Proof of Theorem \ref{thm_pc>0}]
    We use the fact that for any increasing function $f: \N \to \R_+$ such that $\E[f(R)] < \infty$, we have
    \begin{equation}\label{E[f(R)]}
        \sum_{n \geq 1} f(n) \nu(n) = f(0) + \frac{1}{p}\sum_{n \geq 0} [f(n+1) - f(n)] (1-q_n).
    \end{equation}
    This is obtained by noticing that $\nu(n) = \dfrac{(1-q_{n-1})-(1-q_n)}{p}$.
    Applying \eqref{E[f(R)]} with $f=\varphi$ gives
    $$p\E[\varphi(R)]=\Delta \sum_{n \geq 0} c_{n+1}c_{n+1}^\top (1-q_n).$$
    which is less than 1 as soon as $p<\E[\varphi(R)]^{-1}$.
    Then, part 1 of Theorem \ref{boolean_perco} implies that there is almost surely no percolation for such a $p$, hence the lower bound on $p_c(G, \nu)$.
\end{proof}

\begin{proof}[Proof of Theorem \ref{thm_expo_decay}]
    We use the same identity \eqref{E[f(R)]} with $f=\psi_t$, where $t>0$ is such that $\E[\psi_t(R)] < \infty$.
    It gives
    $$p\E[\psi_t(R)] = \sum_{n \geq 0} c^\top_{n+1} e^{tc_{n+1}} (1-q_n),$$
    so that part 2 of Theorem \ref{boolean_perco} gives the exponential decay.

    Then, since $e^{2t \Delta^{n+1}}-e^{2t \Delta^n} =e^{2t \Delta^n} \left(e^{2t \Delta^n(\Delta-1)}-1\right) \geq e^{2t \Delta^n} \Delta^n 2t(\Delta-1)$, we obtain
    \begin{align*}
        p\E[\psi_t(R)] &\leq \sum_{n \geq 0} 2\Delta^n e^{2t\Delta^n} (1-q_n) \\
        &\leq \dfrac{1}{t(\Delta-1)} \sum_{n \geq 0} \left(e^{2t \Delta^{n+1}}-e^{2t \Delta^n}\right)(1-q_n)\\
        &\leq \dfrac{p}{t(\Delta-1)} \E\left[e^{2t\Delta^R}\right],
    \end{align*}
    where in the last inequality we apply \eqref{E[f(R)]} with $f(n)=e^{2t\Delta^n}$.
    Thus, $\E\left[e^{2t\Delta^R}\right] < \infty$ implies that $\E[\psi_t(R)]<\infty$.
\end{proof}

\section{Proof of Theorem \ref{boolean_perco}}\label{Proof of Theorem}
The proof of Theorem \ref{boolean_perco} will be based on an exploration algorithm of the wet connected component of the root, which we describe informally here.
First, we discover all the wet balls containing $\rho$.
This forms our first layer of vertices contained in $W_\rho$.
Then, assuming that the $n$-th layer has been constructed, the $(n+1)$-th layer consists of the union of wet balls whose boundary contains some vertex of the $n$-th layer.
Note that it is not necessary for a wet ball to share a vertex with the previous layer in order to be added to $W_\rho$, only the existence of an edge (in any direction) between the two is required.
If it is finite, $W_\rho$ will be completely revealed after finding an empty layer (otherwise the algorithm never ends).
To prove finiteness and exponential decay, we will compare the size of the layers with a subcritical Galton-Watson process.

The randomness of the appearance of the balls in \eqref{W(G)} is supported by their center.
Indeed, a given ball is said to be a wet ball if it is equal to $B_{R_x}(x)$ for some $x \in V$.
To obtain a domination by a Galton-Watson process, it requires more independence; more precisely, it would be preferable for the events that a given ball is a wet ball to be independent.
This can be achieved by realising the model as a Poisson Point Process (PPP) of balls.
Furthermore, since the balls are discovered during the exploration by their boundary, we want to place the randomness of the appearance of a ball at its boundary.
In Section \ref{Realisation of the model as a PPP}, we first realise the model as a PPP on $V \times \N^*$, where the randomness will still come from the centers of the balls.
Then, in Section \ref{Realisation of the model as PPP of marked balls} we transfer the randomness to the boundary by introducing PPP of marked balls.
Finally, we conclude the proof of Theorem \ref{boolean_perco} in Section \ref{Comparison with a Galton-Watson process} by showing domination by a subcritical Galton-Watson process.

\subsection{Realisation of the model as a PPP on $V \times \N^*$}\label{Realisation of the model as a PPP}
We first define a version of the model by introducing a PPP of a well chosen intensity.
Let $\nu_1$ be the measure on $V \times \N^*$ such that for every $(x,n) \in V \times \N^*$, $\nu_1(\{(x,n)\}) = \lambda_n:= \log(\frac{q_n}{q_{n-1}})$.
Let $\omega_1$ be a PPP of intensity $\nu_1$, i.e. for every $(x,n) \in V \times \N^*$, if we set $\omega_1(x,n):=|\omega_1 \cap \{(x,n)\}|$, then $\omega_1(x,n)$ has Poisson distribution of mean $\lambda_n$ and the $\omega_1(x,n)$ for $(x,n) \in V \times \N^*$ are independent.

Then, for every $x\in V$, let $Y_x:=\sup\{n\geq 1 : \omega_1(x,n) >0\}$ with the convention $Y_x=0$ if it is empty.
Thus, we have that 
\begin{align*}
    \P(Y_x \leq n) &= \P(\forall k>n, \omega_1(x,k) = 0)\\
    &= \prod_{k>n} e^{- \lambda_k}\\
    &= \prod_{k>n} \dfrac{q_{k-1}}{q_k} = q_n.
\end{align*}
Moreover, the $(Y_x)_{x \in V}$ are independent because they are functions of disjoint $\omega_1(x,n)$.
Note that here the $Y_x$ are distributed as the $\sigma_x R_x$ of the previous section. In particular, we have $\P(Y_x>0) = 1-q_0=p$.
Hence, the $(Y_x)_{x\in V}$ define a realisation of the Boolean percolation model we saw above, with wet set 
\begin{equation*}\label{wet_set}
    W= \bigcup_{x \in V} B_{Y_x}(x),
\end{equation*}
where the union is taken over all the vertices of the graph.

\subsection{Realisation of the model as PPP of marked balls}\label{Realisation of the model as PPP of marked balls}
Let $\B = \{B_r(x), x\in V, r\in \N^*\}$ be the collection of balls of the graph $G$. 
There is a canonical surjection from $V \times \N^*$ to $\B$ given by
\begin{equation*}
    \tilde \Phi : \fonction{V \times \N^*}{\B}{(x,r)}{B_r(x)}.
\end{equation*}
Unfortunately, $\tilde \Phi$ is not a bijection, because we could have, for example, $B_n(x) = B_n(y)$ for different $x$ and $y$.
This is why we need to work with balls with marked center, and define $\B^\bullet = \{(B_r(x),x), x\in V, r \in \N^*\}$.
For $B^\bullet =(B,x) \in \B^\bullet$, let $\rad(B^\bullet) = \sup_{y \in B} d(x,y)+1$ be the radius of the ball.
We will use the following convention: if $B^\bullet$ is an element of $\B^\bullet$, $B$ will always denote its first coordinate (i.e. the corresponding ball of $\B$).

We now have a canonical surjection from $V \times \N^*$ to $\B^\bullet$ given by
\begin{equation*}
    \Phi : \fonction{V \times \N^*}{\B^\bullet}{(x,r)}{(B_r(x),x)}.
\end{equation*}
It is still not a bijection because we have $B_r(x) = B_{r+1}(x)$ if $\partial^+ B_r(x) = \emptyset$.
However, we can now easily write $\Phi^{-1}(\{B^\bullet\})$ for any $B^\bullet =(x,B) \in \B$:
we have $\Phi^{-1}(\{B^\bullet\}) = \{x\} \times  \llbracket \rad(B^\bullet), \infty\llbracket$ if $\partial^+B = \emptyset$, and otherwise $\Phi^{-1}(\{B^\bullet\}) = (x, \rad(B^\bullet))$.

Now, we can pushforward $\omega_1$ by $\Phi$, to get a PPP on $\B^\bullet$ of intensity $\nu_2 = \Phi_* \nu_1$, i.e.
$$\nu_2(\{B^\bullet\}) = \cas{\sum_{n \geq \rad(B^\bullet)} \lambda_n}{\si \partial^+B= \emptyset}{\lambda_{\rad(B^\bullet)}}{\text{ otherwise}} \text{ for all }B^\bullet \in \B^\bullet.$$

As stated before, we aim to put the randomness on the boundary of the balls.
This is why we will also consider balls with a marked vertex on their boundary.
Let $\B^\bullet_\bullet= \{(B_r(x),x,y), x\in V, r \in \N^*, y\in \partial B_r(x)\}$ and define $\nu_3$ a measure on $\B^\bullet_\bullet$ by
$$\nu_3(\{B^\bullet_\bullet\}) = \cas{\sum_{n \geq \rad(B^\bullet_\bullet)} \lambda_n}{\si \partial^+B= \emptyset}{\lambda_{\rad(B_\bullet^\bullet)}}{\text{ otherwise}} \text{ for all }B^\bullet_\bullet=(B,x,y) \in \B^\bullet_\bullet,$$
where $\rad(B^\bullet_\bullet)$ is defined in the same way as $\rad(B^\bullet)$.

By defining more precisely the exploration layers mentioned above,
we now construct a coupling $(\omega_2, \omega_3)$ of PPP of intensity respectively $\nu_2$ and  $\nu_3$.
For this, we will need an auxiliary total order on $V$, which is chosen arbitrarily.

First, for all $(B,x,y) \in \B^\bullet_\bullet$ such that $\rho \in  B$, simply let $\omega_3\big((B,x,y)\big)$ be independent Poisson variables of parameters $\nu_3(\{(B,x,y)\})$,
and let
$$\omega_2\big((B,x)\big) = \omega_3\big((B,x,y_{\min})\big),$$
where $y_{\min}$ is the minimal element of $\partial B$ for the auxiliary order of the vertices.
Let $$\ds \C_1 = \bigcup_{\omega_2(\B^\bullet)>0 \atop \rho \in B}B.$$

Then, successively for every $n\geq 1$, construct the $\omega_3\big((B,x,y)\big)$ of all $(B,x,y) \in \B^\bullet_\bullet$
such that $\partial B \cap \C_n \neq \emptyset$ and that were not previously defined:
independently of the previous draws, let $\omega_3\big((B,x,y)\big)$ be independent Poisson variables of parameters $\nu_3(\{(B,x,y)\})$, and let
$$\omega_2\big((B,x)\big) = \omega_3\big((B,x,y_{\min})\big),$$
where $y_{\min}$ is the minimal element of $\partial B \cap \C_n$ for the auxiliary order.
Let $$\C_{n+1} = W(\C_n) - \bigcup_{k\leq n} \C_k,$$
where for every $C \subset V$, $$W(C):= \bigcup \{B \mid \omega_3\big((B,x,y)\big)>0 \text{ for some } x\in V \et y\in C \}.$$

Finally, for all the $\omega_3\big((B,x,y)\big)$ that have not yet been defined, draw again independently Poisson variables of parameter $\nu_3(\{(B,x,y)\})$,
and simply take
$$\omega_2\big((B,x)\big) =\omega_3\big((B,x,y_{\min})\big)$$
as in the first step, $y_{\min}$ being the minimal element of $\partial B$.

Therefore, it appears that $\omega_2$ and $\omega_3$ are respectively distributed as PPP of intensity $\nu_2$ and $\nu_3$.
Since $\omega_2$ has the same distribution as $\Phi(\omega_1)$, it can be coupled with $\omega_1$ in such a way that a.s.  $\Phi(\omega_1)=\omega_2$.

Now, rewrite the wet set $W$ in terms of this new PPP of marked balls.
\begin{lemme}
    Almost surely for the coupling $(\omega_1, \omega_2, \omega_3)$, we have
    \begin{equation}\label{eq:W_with_omega}
        W= \bigcup_{B^\bullet \in \B^\bullet, \omega_2(B^\bullet)>0} B.
    \end{equation}
    
\end{lemme}

\begin{proof}
The coupling have been defined in such a way that a.s. for all $B^\bullet \in \B^\bullet$, we have
$$\sum_{(x,n)\in \Phi^{-1}(\{B^\bullet\})} \omega_1(x,n) = \omega_2(B^\bullet).$$
Let us check both inclusions of \eqref{eq:W_with_omega}.
First, if $y\in W$, then there exists $x\in V$ such that $y \in B_{Y_x}(x)$ i.e. $d(x,y) < Y_x$.
By definition of $Y_x$, it means that there exists $n>d(x,y)$ such that $\omega_1(x,n)>0$. Let $B^\bullet=\Phi(x,n)$; then $y\in B$ and $\omega_2(B^\bullet) \geq \omega_1(x,n)>0 $.

On the other hand, suppose $y\in B$ for some $B^\bullet$ with $\omega_2(B^\bullet)>0$; thus, there exists $(x,n)\in \Phi^{-1}(\{B^\bullet\})$ such that $\omega_1(x,n)>0$. Since $y\in B$, it implies that $d(x,y)<n \leq Y_x$, so $y\in B_{Y_x}(x) \subset W$.
\end{proof}

Then, notice that the wet connected component of the root is entirely discovered by our exploration.
\begin{lemme}\label{inclusion_of_W_0}
    Almost surely for the coupling $(\omega_1, \omega_2, \omega_3)$, we have
    $$ W_\rho \subset \bigcup_{n\geq 1} \C_n.$$
\end{lemme}

\begin{proof}
    Let $z\in W_\rho$.
    According to the identity \eqref{eq:W_with_omega}, there exists a chain of balls $B_1^\bullet, \dots, B_k^\bullet$ such that $\rho \in B_1$ and $z\in B_k$, and $$\forall i \in \{1, \dots,k\},\ \omega_2(B_i^\bullet)>0 \text{ and for }  i<k,\  B_i \cap (B_{i+1} \cup \partial B_{i+1}) \neq \emptyset.$$
    Furthermore, we can assume that $B_j \not \subset B_i$ for all $i\neq j$, i.e. no ball of the chain is included in a different ball.
    We argue that this implies $B_i \cap \partial B_{i+1} \neq \emptyset$ for all $i$.
    In fact, if there exists $y \in B_i \cap B_{i+1}$, take $x \in B_i - B_{i+1}$; since $B_i$ is connected, there is a path from $y$ to $x$ that stays inside $B_i$ ; but $x \notin B_{i+1}$ so the path necessarily leaves $B_{i+1}$, so it intersects $\partial B_{i+1}$; hence, it provides a vertex in $B_i \cap \partial B_{i+1}$.
    
    Assume that $B^\bullet_k$ is not discovered in the exploration, so $B_k \not \subset \bigcup_{n\geq 1} \C_n$.
    Therefore, there exists a first ball $B_{i+1}^\bullet$ of the chain such that $B_{i+1} \not \subset \bigcup_{n\geq 1} \C_n$. This first ball cannot be $B_1^\bullet$ since $B_1 \subset \C_1$ by definition of $\C_1$.
    By assumption on the chain, $\partial B_{i+1}$ intersects $B_i$, so the reunion $\bigcup_{n\geq 1} \C_n$ since it contains $B_i$.
    It means that $\C_n \cap \partial B_{i+1} \neq \emptyset$ for some $n\geq 1$ (take $n$ minimal for this property).
    According to the construction of the coupling, this induces
    $$\sum_{y \in \partial B_{i+1} \cap \C_n} \omega_3\big((B_{i+1}^\bullet, y)\big) \geq \omega_2(B^\bullet_{i+1})>0.$$
    Thus, $\omega_3\big((B_{i+1}^\bullet, y)\big)$ is positive for some $y\in \partial B_{i+1} \cap \C_n$, so $B_{i+1} \subset W(\C_n) \subset \bigcup_{n\geq 1} \C_n$, which yields a contradiction. Hence $z\in B_k \subset \bigcup_{n\geq 1} \C_n$.
\end{proof}

In the last paragraph, we show that $(|\C_n|)$ is dominated by a subcritical Galton-Watson process.
Together with the result of Lemma \ref{inclusion_of_W_0}, it concludes the proof of Theorem \ref{boolean_perco}.

\subsection{Comparison with a Galton-Watson process}\label{Comparison with a Galton-Watson process}

For every vertex $y\in V$, define the random variable
$$\xi_y := \sum_{B^\bullet_\bullet = (B, \cdot, y)} |B| \ind{\omega_3(B^\bullet_\bullet) >0},$$
which is an upper bound for the number of vertices in $W(\{y\})$.
The $\xi_y$ for $y\in V$ are independent since they concern disjoint $\omega_3(\cdot)$.
For the root, we also define
$$\tilde \xi_\rho := \sum_{B^\bullet, \rho \in B} |B| \ind{\omega_2(B^\bullet) >0},$$
which is an upper bound for $|\C_1|$.
Both the $\xi_y$ and $\tilde \xi_\rho$ are stochastically dominated by the same distribution, that we introduce in the following lemma.
Then, we show that $(|\C_n|)$ is dominated by a Galton-Watson process that admits this distribution as reproduction law.

\begin{lemme}\label{lemma_domination}
    Let $(Z^{j,r})_{j\geq 1, r\geq 1}$ be a collection of independent Poisson random variables such that for all $j\geq 1$, $r \geq 1$, $Z^{j,r}$ has mean $\sum_{n \geq r} \lambda_n$.
    Set 
    \begin{equation}\label{def_xi}
        \xi:= \sum_{r\geq 1} c_r \sum_{j=1}^{\Delta  c_r^\top} \ind{Z^{j,r}>0}.
    \end{equation}
    Then, for all $y \in V$, we have the stochastic dominations
    $$\xi_y \preceq \xi \et \tilde \xi_\rho \preceq \xi.$$
\end{lemme}

\begin{proof}
    First, observe that if $B^\bullet_\bullet \in \B^\bullet_\bullet$ has radius $r\geq 1$, we have $|B| \leq c_r$.
    Then, for every $y \in V$,
    $$\xi_y \leq \sum_{r\geq 1} c_r \sum_{B^\bullet_\bullet = (\cdot, \cdot, y)} \ind{\omega_3(B^\bullet_\bullet)>0} \ind{\rad(B^\bullet_\bullet) =r}.$$
    The $\omega_3(B^\bullet_\bullet)$ that appears in the above sum are independent Poisson variables whose parameters depend on $B^\bullet_\bullet$. Since we are only looking at events of the form $\{\omega_3(B^\bullet_\bullet)>0\}$, we get an upper bound by considering the largest possible parameter. More precisely, we can construct a coupling of $\omega_3$ with $\omega_3'$, where
    the $\omega_3'(B^\bullet_\bullet)$ are independent Poisson variables of parameters $\sum_{n \geq \rad(B^\bullet_\bullet)} \lambda_n$ and such that a.s. for all $B^\bullet_\bullet \in \B^\bullet_\bullet$, $\omega_3(B^\bullet_\bullet) \leq \omega'_3(B^\bullet_\bullet)$. Then,
    \begin{equation}\label{dom_xi_y}
        \xi_y \leq \sum_{r\geq 1} c_r \sum_{B^\bullet_\bullet = (\cdot, \cdot, y)} \ind{\omega_3'(B^\bullet_\bullet)>0} \ind{\rad(B^\bullet_\bullet) =r}.
    \end{equation}
    For every $r\geq 1$, the marked balls $B^\bullet_\bullet \in \B^\bullet_\bullet$ with $\rad(B^\bullet_\bullet) =r$ that contribute in \eqref{dom_xi_y} are of the form $(B_r(x), x, y)$ for some $x \in V$ such that $y \in \partial B_r(x)$.
    Then, $y \in \partial B_r(x)$ only if there is a neighbour of $y$ at distance less than $r$ from $x$; there are at most $\Delta$ neighbours of $y$, and $c^\top_r$ possible $x$ for each of them.
    Therefore, there are at most $\Delta c^\top_r$ non-zero terms in the second sum in \eqref{dom_xi_y}, hence we get the domination by the law of $\xi$.
    We can get an explicit coupling between the two by taking, for every $r\geq 1$, the $(Z^{j,r})_{j\geq 1}$ to be the $\omega_3'(B^\bullet_\bullet)$ of the balls $B^\bullet_\bullet$ with radius $\rad(B^\bullet_\bullet) =r$, arranged in an arbitrary order, and other independent Poisson variables to complete the sequence.

    The proof for $\tilde \xi_\rho$ works similarly.
    Introduce $\omega_2'$ a collection of independent Poisson variables indexed by $\B^\bullet$ such that for all $B^\bullet \in \B^\bullet$, $\omega_2'(B^\bullet)\geq \omega_2(B^\bullet)$ a.s. and $\omega_2'(B^\bullet)$ has mean $\sum_{n\geq \rad(B^\bullet)} \lambda_n$.
    Thus,
    $$\tilde \xi_\rho \leq \sum_{r\geq 1} c_r \sum_{B^\bullet, \rho \in B} \ind{\omega_2'(B^\bullet)>0}\ind{\rad(B^\bullet) =r}.$$

    The second sum contains at most $c_r^\top$ non-zero terms, since the number of vertices $x\in V$ satisfying $\rho \in B_r(x)$ is bounded by $|B_r^\top(\rho)| \leq c_r^\top \leq \Delta c_r^\top$.
    Hence, the stochastic domination of the law of $\tilde \xi_\rho$ by the law of $\xi$.
\end{proof}

Let $(\xi_k^{(n)})_{k,n \in \N}$ be a collection of independent copies of $\xi$, $Z_0 =1$ and for every $n \in \N$, $\ds Z_{n+1}=\sum_{k=1}^{Z_n} \xi_k^{(n)}$.
Then, $(Z_n)$ is a Galton-Watson process of reproduction law the law of $\xi$.
We construct a version of $(Z_n)$ that satisfies $|\C_n| \leq Z_n$ almost surely.

First, we know by Lemma \ref{lemma_domination} that we can choose $\xi_1^{(0)}$ so that $|\C_1| \leq \xi_1^{(0)}$ a.s., by taking
$$\xi_1^{(0)} = \sum_{r\geq 1} c_r \sum_{j=1}^{\Delta c_r^\top} \ind{Z(j,r)^{(0)}_1 >0},$$
where for every $r\geq 1$, the $Z(j,r)^{(0)}_1$ are taken to be the $\omega_2'(B^\bullet)$ (defined in the proof of Lemma \ref{lemma_domination}) in arbitrary order with $\rad(B^\bullet) = r$ as well as others independent Poisson variables of parameter $\sum_{n \geq r} \lambda_n$. 

Then, assuming the coupling is defined until step $n$, i.e. we have $|\C_m| \leq Z_m$ for all $m\leq n$, construct $Z_{n+1}$ in the following way.
Call $y_1, \dots, y_{|\C_n|}$ the vertices of $\C_n$ arranged in arbitrary order.
We have $|\C_{n+1}| \leq \xi_{y_1} +\dots+ \xi_{y_{|\C_n|}}$.

For all $1 \leq k \leq |\C_n|$, let
$$\xi_k^{(n)} = \sum_{r\geq 1} c_r \sum_{j=1}^{\Delta c_r^\top} \ind{Z(j,r)^{(n)}_k >0},$$
where for every $r\geq 1$, the $Z(j,r)^{(n)}_k$ are taken to be the $\omega_3'(B^\bullet_\bullet)$ (see the proof of Lemma \ref{lemma_domination}) in arbitrary order for $\B_\bullet^\bullet$ of the form $(B, \cdot, y_k)$ with $\rad(B^\bullet_\bullet) = r$, as well as others independent Poisson variables of parameter $\sum_{n \geq r} \lambda_n$.
With this choice of $\xi_k^{(n)}$, we have a.s. $\xi_{y_k} \leq \xi_k^{(n)}$.
Moreover, it is the case that the $(\xi_k^{(n)})_{k\geq 1}$ are independent, and independent of the $\xi_k^{(m)}$ for $m<n$ because the $\omega_3'(B^\bullet_\bullet)$ collected at this step have never been revealed before; here we use the fact that the marked balls of $\B^\bullet_\bullet$ are attached to some vertex of their boundary and that the $\C_m$ are disjoint.
Note that they are also independent from the $\omega_2'(B^\bullet)$ present in the first layer, because no ball containing $\rho$ can be present in higher layers.

Therefore, we have
$$|\C_{n+1}| \leq \xi_{y_1} +\dots+ \xi_{y_{|\C_n|}} \leq \xi_1^{(n)} + \dots + \xi_{|\C_n|}^{(n)} \leq \xi_1^{(n)} + \dots + \xi_{Z_n}^{(n)} =Z_{n+1}, $$
where the $\xi_{k}^{(n)}$ for $k>|\C_n|$ are i.i.d. copies of $\xi$.

This yields a coupling such that almost surely $|\C_n| \leq Z_n$ for all $n\geq 1$.
We conclude the proof of Theorem \ref{boolean_perco} using well-known properties of the Galton-Watson processes.

Recalling the definition of $\xi$ from \eqref{def_xi},
we have $$\P(Z^{r,j}>0) = 1-e^{-\sum_{k\geq r} \lambda_k} = 1-\prod_{k\geq r} e^{-\lambda_k} = 1-q_{r-1},$$ hence
$$\E[\xi] = \sum_{r\geq 0} c_{r+1} \Delta c_{r+1}^\top (1-q_r).$$
It is well known that if $\E[\xi] < 1$, the Galton-Watson process $(Z_n)$ extincts almost surely.
Therefore, we also get under that condition that almost surely $\C_n$ is empty for large enough $n$, so $W_\rho$ is finite.
Since the graph is countable and $\rho$ was chosen arbitrarily, then almost surely the wet set $W$ does not contain any infinite connected component.
This proves part 1. of Theorem \ref{boolean_perco}.

For the second part, we use the fact (see \cite{LyonsPeres} Exercise 5.33) that if the reproduction law of a subcritical Galton-Watson process admits finite exponential moment, then the total size of the process has an exponential tail.
For all $t>0$,
\begin{align*}
    \E[e^{t\xi}] &= \E \left[\prod_{r\geq1}\prod_{j=1}^{\Delta  c_r^\top} e^{t c_r \ind{Z^{r,j}>0}} \right]\\
    &=\prod_{r\geq1}\prod_{j=1}^{\Delta c_r^\top} \E[e^{tc_r \ind{Z^{r,j}>0}}]\\
    &=\prod_{r\geq1}\left(e^{t c_r}(1-q_{r-1})+q_{r-1}\right)^{\Delta c_r^\top}.
\end{align*}
Then, taking the logarithm and applying the inequality $\log(1+x)\leq x$, we get
\begin{align*}
    \log \E[e^{t \xi}] &= \sum_{r\geq 0} \Delta c_{r+1}^\top \log\left(e^{t c_{r+1}}(1-q_r)+q_r\right) \\
    &\leq \sum_{r\geq0} \Delta c_{r+1}^\top e^{t c_{r+1}} (1-q_r).
\end{align*}
By assumption \eqref{condition_expodecay}, there exists $t>0$ such that this sum is finite.
Then, by comparison with the total size of $(Z_n)$, we get the exponential decay of the size of $W_\rho$.

\bibliographystyle{plain}
\bibliography{boolean}

\end{document}